\documentclass[12pt,reqno]{amsart}
\usepackage{amscd}
\usepackage{amsmath}
\usepackage{mathtools, nccmath}
\usepackage{latexsym}
\usepackage{amsfonts}
\usepackage{amssymb}
\usepackage{amsthm}
\usepackage{graphicx}
\usepackage{verbatim}
\usepackage{mathrsfs}
\usepackage{enumerate}
\usepackage[colorlinks=true, linkcolor=blue, citecolor=blue, urlcolor=blue]{hyperref}
\usepackage{cleveref}
\usepackage{xcolor}

\newtheorem{theorem}{Theorem}[section]
\newtheorem{lemma}[theorem]{Lemma}
\newtheorem{corollary}[theorem]{Corollary}
\newtheorem{proposition}[theorem]{Proposition}

\theoremstyle{remark}
\newtheorem{remark}[theorem]{Remark}

\newcommand{\beq}{\begin{equation}}
\newcommand{\eeq}{\end{equation}}

\newcommand{\R}{\mathbb{R}}
\newcommand{\Rd}{{\mathbb{R}^d}}

\newcommand{\inv}{^{-1}}
\newcommand{\bbP}{\mathbb{P}}

\newcommand{\Lpnorm}[2]{\| #1 \|_{L^{#2}}}

\newcommand{\one}{^{(1)}}
\newcommand{\two}{^{(2)}}

\newcommand{\Hsnorm}[2]{\| #1 \|_{H^{#2}}}
\newcommand{\du}{\delta u}
\newcommand{\dtau}{\delta \tau}
\newcommand{\divg}{\nabla \cdot}

\newcommand{\ipltwo}[2]{\left( #1, #2\right)_{L^2}}

\newcommand{\ipHs}[3]{\left( #1, #2\right)_{H^{#3}}}

\newcommand{\EE}{\mathcal{E}}
\newcommand{\YY}{\mathcal{Y}}
\newcommand{\ZZd}{\mathcal{Z}}
\newcommand{\DD}{\mathcal{D}}

\newcommand{\LL}{\mathcal{L}}

\allowdisplaybreaks

\begin{document}

\title[Regularity of solution maps]{Regularity of the solution map for the Oldroyd--B model with fractional stress dissipation}

\author[Gerard Misio{\l}ek, Jiahong Wu and Mingyu Yu]
{Gerard Misio{\l}ek$^{1}$, Jiahong Wu$^{2}$ and Mingyu Yu$^{3}$}

\address{$^1$
	Department of Mathematics,
	University of Notre Dame,
	Notre Dame, IN 46556, USA}
\email{gmisiole@nd.edu}

\address{$^2$
	Department of Mathematics,
	University of Notre Dame,
	Notre Dame, IN 46556, USA}
\email{jwu29@nd.edu}

\address{$^3$
	Department of Mathematics,
	University of Notre Dame,
	Notre Dame, IN 46556, USA}
\email{myu4@nd.edu}

\date{\today}

\vskip .2in
\begin{abstract}
The data-to-solution map of the incompressible Euler equations is continuous but not uniformly continuous on Sobolev spaces, whereas for the Navier--Stokes equations viscosity makes the map Lipschitz continuous on small data, uniformly in time. This paper identifies the behavior of the solution map for a hybrid system sitting between these two extremes: the incompressible Oldroyd--B model in $\R^d$, $d=2,3$, in the high-Reynolds-number, high-Weissenberg-number regime, in which the velocity equation is a forced Euler equation and the only dissipation is a fractional diffusion $\eta(-\Delta)^\beta$, $\frac12\le\beta\le1$, acting on the stress tensor. We prove that, on the ball of small $H^s$ data ($s>1+\frac d2$) on which global solutions exist, the solution map is Lipschitz continuous with respect to the $H^{r}$-norm, $r=s-(2\beta-1)$, uniformly for all time, and Lipschitz into the natural dissipation spaces; the constant is explicit and tends to $\sqrt3$ as the size of the data tends to zero. Consequently the map is H\"older continuous into $C_b([0,\infty);H^{s'})$ for every $s'<s$, and is Lipschitz in $H^s$ itself when $\beta=\frac12$. Thus, although the velocity field is only transported, the solution map behaves like that of the Navier--Stokes equations rather than that of the Euler equations. The proof exploits the damped wave structure generated by the coupling between $u$ and $\nabla\cdot\tau$ through a Lyapunov functional with a cross term $(\delta u,\nabla\cdot\delta\tau)_{H^{r-\beta}}$, which transfers the stress dissipation to the velocity difference, together with a low-frequency interpolation argument that keeps all Gronwall coefficients integrable on $(0,\infty)$. The loss of $2\beta-1$ derivatives is exactly the gap between the hidden velocity dissipation of order $2(1-\beta)$ and the dissipation required by the Lipschitz estimate, and we show that it cannot be reduced by the energy method.
\end{abstract}

\maketitle

\section{Introduction}
\setcounter{equation}{0}

The regularity of the data-to-solution map is one of the most fundamental and delicate aspects of the well-posedness theory of nonlinear fluid equations, and it separates viscous from inviscid flows in a sharp way. Consider the incompressible flow of a homogeneous fluid in $\R^d$, $d=2,3$,
\begin{equation}
	\label{eq:NSE}
	\begin{cases}
		\partial_t u + u\cdot\nabla u +\nabla p= \nu \Delta u,\qquad
		(x,t)\in \Rd\times \R_+,\\
		\divg u =0,\\
		u(x,0) = u_0(x),
	\end{cases}
\end{equation}
where $u$ is the velocity, $p$ the pressure and $\nu\ge0$ the kinematic viscosity; \eqref{eq:NSE} is the Navier--Stokes equation when $\nu>0$ and the Euler equation when $\nu=0$ (see \cite{MajdaBertozzi,BCD} for comprehensive accounts). For the Navier--Stokes equations with $d=3$, Fujita and Kato \cite{FujitaKato} constructed, for initial data small in the critical space $\dot H^{\frac12}(\R^3)$, a unique global solution in $C([0,\infty);\dot H^{\frac12})\cap L^2(0,\infty;\dot H^{\frac32})$. An important feature of the Fujita--Kato theory, developed in detail in the book of Bahouri, Chemin and Danchin \cite[Chapter 5]{BCD}, is that the solution map is Lipschitz continuous \emph{uniformly in time}: if $u$ and $v$ are the solutions emanating from two small data $u_0$ and $v_0$, then
\begin{equation}\label{eq:FK-stab}
\sup_{t\ge 0}\,\|u(t)-v(t)\|_{\dot H^{\frac12}}
+\Big(\int_0^\infty \|u(t)-v(t)\|_{\dot H^{\frac32}}^2\,dt\Big)^{\frac12}
\le C\,\|u_0-v_0\|_{\dot H^{\frac12}} ,
\end{equation}
with $C$ depending only on $\nu$ and the size of the data. Viscosity damps the high frequencies of the difference $u-v$ and produces the space-time integrability that makes the Gronwall argument behind \eqref{eq:FK-stab} close on the whole half-line.

The situation for the Euler equations is entirely different. Himonas and Misio{\l}ek \cite{AlexGerard} proved that the solution map of the incompressible Euler equations is continuous but \emph{not} uniformly continuous from bounded subsets of $H^s(\R^d)$ into $C([0,T];H^s(\R^d))$, $s>1+\frac d2$, no matter how small $T>0$ is; Inci \cite{Inci} subsequently showed that the fixed-time map $u_0\mapsto u(T)$ is nowhere locally uniformly continuous and nowhere Fr\'echet differentiable on its natural domain. The mechanism behind this loss of regularity is the transport structure of the Euler equations, which allows nearby data to be stretched apart by the flow at a rate that cannot be controlled uniformly. We stress that the failure of uniform continuity in \cite{AlexGerard,Inci} concerns the $H^s$ topology, in which the solutions themselves live: if the difference of two $H^s$ solutions is measured one derivative lower, in $H^{s-1}$, then the classical energy estimate for the difference closes and the solution map of the Euler equations is Lipschitz from bounded subsets of $H^s$ into $C([0,T];H^{s-1})$ (see, e.g., \cite[Chapter 7]{BCD}). This Lipschitz bound, however, costs a full derivative and, because the Gronwall coefficient $\|\nabla u\|_{L^\infty}$ is not integrable on $(0,\infty)$, it holds only on bounded time intervals, with a constant that grows exponentially in $T$.

\medskip
The purpose of this paper is to determine which of these two behaviors is exhibited by a hybrid system in which the velocity equation is inviscid but a second, coupled unknown carries a (fractional) dissipation. The model we consider is the incompressible Oldroyd--B system for viscoelastic fluids in the high-Reynolds-number, high-Weissenberg-number regime,
\begin{equation}
	\label{eq:OB_eta}
	\begin{cases}
		\partial_t u+u\cdot\nabla u+\nabla p=\nabla\cdot\tau,\qquad
		x\in\mathbb{R}^d,\ t>0,\\
		\partial_t \tau+u\cdot\nabla\tau+\eta(-\Delta)^\beta\tau
		+Q(\tau,\nabla u)=D(u),\\
		\nabla\cdot u=0,\\
		u(x,0)=u_0(x),\quad \tau(x,0)=\tau_0(x),
	\end{cases}
\end{equation}
where $u=u(x,t)$ is the velocity field, $p=p(x,t)$ the scalar pressure and $\tau=\tau(x,t)$ the non-Newtonian (polymeric) stress, a symmetric $d\times d$ matrix. Here $D(u)$ and $W(u)$ denote the symmetric and antisymmetric parts of the velocity gradient,
\[
D(u)=\tfrac{1}{2}\bigl(\nabla u+(\nabla u)^\top\bigr),\qquad
W(u)=\tfrac{1}{2}\bigl(\nabla u-(\nabla u)^\top\bigr),
\]
and, for a fixed constant $b\in[-1,1]$, the bilinear term
\[
Q(\tau,\nabla u)=\tau W(u)-W(u)\tau-b\bigl(D(u)\tau+\tau D(u)\bigr)
\]
encodes the objective (frame-indifferent) rotation and stretching of the elastic stress by the flow. The fractional Laplacian $(-\Delta)^\gamma$ is defined through the Fourier transform, $\widehat{(-\Delta)^\gamma f}(\xi)=|\xi|^{2\gamma}\hat f(\xi)$, and we write $\Lambda=(-\Delta)^{\frac12}$. The parameters $\eta>0$ and $\frac12\le\beta\le1$ are constants; $\beta=1$ is the standard diffusive-stress model, while $\beta<1$ describes nonlocal diffusion in the stress.

The Oldroyd--B model is one of the most widely used constitutive systems for incompressible viscoelastic fluids, typically a Newtonian solvent carrying long elastic polymer chains. It can be derived as an exact moment closure of a Navier--Stokes--Fokker--Planck system for dilute suspensions of Hookean dumbbells; the stress $\tau$ is (up to a multiple of the identity absorbed in the pressure) the deviation of the conformation tensor from equilibrium. The general model contains a viscous term $\nu(-\Delta)^\alpha u$ in the momentum equation and a damping term $a\tau$ in the stress equation, with $\nu\sim 1/\mathrm{Re}$ and $a\sim1/\mathrm{We}$, where $\mathrm{Re}$ and $\mathrm{We}$ are the Reynolds and Weissenberg numbers. The system \eqref{eq:OB_eta} is the limiting case $\mathrm{Re}=\mathrm{We}=\infty$: the solvent is inviscid, the stress relaxation is absent, and the sole regularizing mechanism is the wave-number dependent dissipation $\eta(-\Delta)^\beta\tau$. This is precisely the regime that is most challenging both numerically and analytically, and the one where the interaction between the velocity and the stress plays the decisive role.

\medskip
Even the most fundamental questions for \eqref{eq:OB_eta}, such as the global existence of solutions and their stability, are far from trivial. The two equations are tightly coupled: the velocity is driven by $\nabla\cdot\tau$, while the stress is driven by $\nabla u$ through $D(u)$ and $Q(\tau,\nabla u)$. Without the stress, the velocity equation is the Euler equation, whose solutions can develop small scales extremely rapidly. In two dimensions the gradient of the vorticity $\omega=\nabla\times u$ obeys
\[
\|\nabla \omega(t)\|_{L^p}
\lesssim
\bigl(\|\nabla \omega_0\|_{L^p}\bigr)^{\exp(\|\omega_0\|_{L^\infty}t)},
\qquad 2 \le p \le \infty,
\]
and Kiselev and \v{S}ver\'ak \cite{KiselevSverak} showed that this double-exponential rate is actually attained (see also \cite{Denisov,Zlatos}). In three dimensions, Elgindi \cite{Elgindi} proved that $C^{1,\alpha}$ solutions of the Euler equations can blow up in finite time (see also \cite{EGM}), and Chen and Hou \cite{ChenHou} established finite-time blow-up for smooth data in the presence of a boundary. In the absence of velocity dissipation there is therefore no reason, a priori, to expect the $H^s$-norm of $u$ to remain bounded, let alone to expect a stable dependence on the data; whatever control one obtains must come from the stress equation through the coupling.

That such control is available was discovered by Constantin, Wu, Zhao and Zhu \cite{CWZZ}. Applying the Leray projection $\bbP=I-\nabla\Delta\inv\divg$ and using $\bbP\nabla\cdot D(u)=\frac12\Delta u$ for divergence-free $u$, they observed that any solution of \eqref{eq:OB_eta} also solves the system of damped wave equations
\begin{equation}
	\label{eq:wave-nonlinear}
	\begin{cases}
		\partial_t^2 u+\eta(-\Delta)^\beta\partial_t u-\tfrac12\Delta u
		= N_1,\\[3pt]
		\partial_t^2\, \bbP\nabla\cdot\tau+\eta(-\Delta)^\beta\partial_t\,\bbP\nabla\cdot\tau-\tfrac12\Delta\,\bbP\nabla\cdot\tau
		= N_2,
	\end{cases}
\end{equation}
where the nonlinear terms are
\[
N_1=-\bigl(\partial_t+\eta(-\Delta)^\beta\bigr)\bbP(u\cdot\nabla u)-\bbP\nabla\cdot(u\cdot\nabla\tau)-\bbP\nabla\cdot Q(\tau,\nabla u),
\]
\[
N_2=-\tfrac12\Delta\,\bbP(u\cdot\nabla u)-\partial_t\bbP\nabla\cdot\bigl(u\cdot\nabla\tau+Q(\tau,\nabla u)\bigr).
\]
The structure of \eqref{eq:wave-nonlinear} is remarkable: although the velocity equation in \eqref{eq:OB_eta} contains no dissipation whatsoever, the coupling with $\tau$ endows $u$ with both a damping and a dispersive mechanism. At the level of the Fourier symbol, $\lambda^2+\eta|\xi|^{2\beta}\lambda+\frac12|\xi|^2=0$, the slowly decaying mode at high frequencies behaves like $\lambda\approx -\frac{1}{2\eta}|\xi|^{2(1-\beta)}$, so that the velocity effectively experiences a fractional dissipation of order $2(1-\beta)$. Exploiting this structure through a Lyapunov functional containing the cross term $(u,\nabla\cdot\tau)_{H^{s-\beta}}$, \cite{CWZZ} obtained the small-data global well-posedness and stability of \eqref{eq:OB_eta} for all $\frac12\le\beta\le1$; the threshold $\beta=\frac12$ is the current limit of the theory, and the case $0\le\beta<\frac12$ remains open. The same structure was subsequently used by Wang, Wu, Xu and Zhong \cite{WWXZ}, who represented \eqref{eq:wave-nonlinear} in integral form and derived sharp large-time decay rates for the solutions. Small-data global solutions of \eqref{eq:OB_eta} in critical Besov spaces were obtained by Wu and Zhao \cite{WuZhao}, and Chen, Liu, Qin and Ye \cite{ChenLiuQinYe} treated a class of large data in an $L^p$ framework when $\beta=1$. For the related models in which the stress damping $a\tau$ is kept, small-data global well-posedness without velocity dissipation was established by Elgindi and Rousset \cite{ElgindiFrederic} in two dimensions and by Elgindi and Liu \cite{ElgindiLiu} in three dimensions, and in the fully dissipative case Constantin and Kliegl \cite{ConstKliegl} proved global regularity in two dimensions.

The stability result of \cite{CWZZ} is the starting point of the present work. We state it in the form in which we shall use it.

\begin{proposition}[\cite{CWZZ}]
	\label{prop:small-data}
	Let $d=2,3$, $s>1+\frac{d}{2}$, $\eta>0$ and $\frac12\le\beta\le1$. Assume $(u_0,\tau_0)\in H^s(\R^d)$, $\nabla\cdot u_0=0$ and $\tau_0$ is symmetric. There exist $\varepsilon_*>0$ and $C_*\ge1$, depending only on $d,s,\beta$ and $\eta$, such that if
	\[
	\|u_0\|_{H^s}+\|\tau_0\|_{H^s}\le \varepsilon\le\varepsilon_*,
	\]
	then \eqref{eq:OB_eta} has a unique global solution $(u,\tau)\in C([0,\infty);H^s(\R^d))$, and for all $t\ge0$,
	\begin{equation}\label{eq:small-data-bound}
	\|u(t)\|_{H^s}^2 + \|\tau(t)\|_{H^s}^2
	+ \int_0^{t} \Bigl(\|\nabla u(t')\|_{H^{s-\beta}}^2+\|\Lambda^{\beta}\tau(t')\|_{H^s}^2\Bigr)\,dt'
	\le C_*\,\varepsilon^2 .
	\end{equation}
\end{proposition}

In \cite{CWZZ} the bound \eqref{eq:small-data-bound} is obtained in the form $\|u\|_{H^s}^2+\|\tau\|_{H^s}^2+2\int_0^t(\eta\|\Lambda^\beta\tau\|_{H^s}^2+\frac{k_0}{2}\|\nabla u\|_{H^{s-\beta}}^2)dt'\le C\varepsilon^2$ with a small constant $k_0=k_0(\eta)$; \eqref{eq:small-data-bound} follows by absorbing $\eta$ and $k_0$ into $C_*$. The velocity dissipation $\int_0^\infty\|\nabla u\|_{H^{s-\beta}}^2dt$ in \eqref{eq:small-data-bound} is exactly the hidden dissipation of order $2(1-\beta)$ predicted by \eqref{eq:wave-nonlinear}; it is weaker than the $H^{s+1}$ regularization of the Navier--Stokes equations by $\beta$ derivatives, and it is the only information on $u$ beyond $L^\infty_tH^s$ that we shall use.

\medskip
\noindent\textbf{Main result.}
The main goal of this paper is to understand the behavior of the data-to-solution map of \eqref{eq:OB_eta}: does it behave like the solution map of the Navier--Stokes equations, for which \eqref{eq:FK-stab} holds, or like that of the Euler equations, which is not even uniformly continuous? The velocity equation is an Euler equation, the stress dissipation is only fractional and only partially transmitted to $u$, and the difference of two solutions is subject to the same stretching mechanisms that destroy uniform continuity for the Euler flow. Nevertheless, we show that the solution map is \emph{Lipschitz continuous uniformly for all time}, with a loss of $2\beta-1$ derivatives, in the spirit of \eqref{eq:FK-stab}: the Oldroyd--B system behaves like the Navier--Stokes equations even though its velocity is only transported. Compared with the Euler equations, where a Lipschitz bound is available only in $H^{s-1}$ and only on bounded time intervals, the uniformity in time is new for every $\beta\in[\frac12,1]$; the loss of regularity is strictly smaller than one derivative when $\beta<1$, and there is no loss at all when $\beta=\frac12$, so that the endpoint case behaves exactly like the Navier--Stokes equations. The precise statement is the following; for a pair $(f,g)$ we write $\|(f,g)\|_{H^\rho}:=(\|f\|_{H^\rho}^2+\|g\|_{H^\rho}^2)^{1/2}$.

\begin{theorem}
	\label{thm:unif-dep}
	Let $d=2,3$, $s>1+\frac{d}{2}$, $\eta>0$ and $\frac{1}{2}\le\beta\le 1$, and set
	\[
	r=s-(2\beta-1).
	\]
	There exist $\varepsilon_0\in(0,\varepsilon_*]$ and $C_1>0$, depending only on $d,s,\beta$ and $\eta$, with the following property. Let $0<\varepsilon\le\varepsilon_0$ and let $(u^{(1)},\tau^{(1)})$ and $(u^{(2)},\tau^{(2)})$ be the global solutions of \eqref{eq:OB_eta} given by \Cref{prop:small-data} with initial data satisfying $\|u_0^{(i)}\|_{H^s}+\|\tau^{(i)}_0\|_{H^s}\le\varepsilon$, $i=1,2$. Then, for all $t\ge0$,
	\begin{equation}
		\label{eq:unif-dep}
		\bigl\|\bigl(u^{(1)}(t)-u^{(2)}(t),\,\tau^{(1)}(t)-\tau^{(2)}(t)\bigr)\bigr\|_{H^r}
		\le \sqrt3\, e^{C_1\varepsilon^2}\,
		\bigl\|\bigl(u^{(1)}_0-u^{(2)}_0,\,\tau^{(1)}_0-\tau^{(2)}_0\bigr)\bigr\|_{H^r},
	\end{equation}
	and, in addition, the differences are Lipschitz in the dissipation spaces,
	\begin{equation}
		\label{eq:unif-dep-diss}
		\int_0^\infty\Bigl(\Hsnorm{\nabla(u^{(1)}-u^{(2)})}{r-\beta}^2
		+\Hsnorm{\Lambda^\beta(\tau^{(1)}-\tau^{(2)})}{r}^2\Bigr)dt
		\le C_1\, e^{C_1\varepsilon^2}\,
		\bigl\|\bigl(u^{(1)}_0-u^{(2)}_0,\,\tau^{(1)}_0-\tau^{(2)}_0\bigr)\bigr\|_{H^r}^2 .
	\end{equation}
\end{theorem}

Since $\|\cdot\|_{H^r}\le\|\cdot\|_{H^s}$, the right-hand side of \eqref{eq:unif-dep} may of course be replaced by $\sqrt3\,e^{C_1\varepsilon^2}\,\|(u_0^{(1)}-u_0^{(2)},\tau^{(1)}_0-\tau^{(2)}_0)\|_{H^s}$. Interpolating \eqref{eq:unif-dep} with the uniform $H^s$ bound \eqref{eq:small-data-bound} yields uniform continuity, in fact H\"older continuity, into every $H^{s'}$ with $s'<s$.

\begin{corollary}\label{cor:holder}
	Under the assumptions of \Cref{thm:unif-dep}, let $r\le s'\le s$ and $\theta=\frac{s'-r}{s-r}\in[0,1]$ (with $\theta=0$ if $r=s$). Then for all $t\ge0$,
	\begin{equation*}
	\begin{split}
	&\bigl\|\bigl(u^{(1)}(t)-u^{(2)}(t),\,\tau^{(1)}(t)-\tau^{(2)}(t)\bigr)\bigr\|_{H^{s'}}\\
	&\qquad\le \bigl(2\sqrt{C_*}\,\varepsilon\bigr)^{\theta}\bigl(\sqrt3\,e^{C_1\varepsilon^2}\bigr)^{1-\theta}
	\bigl\|\bigl(u^{(1)}_0-u^{(2)}_0,\,\tau^{(1)}_0-\tau^{(2)}_0\bigr)\bigr\|_{H^r}^{1-\theta}.
	\end{split}
	\end{equation*}
	In particular, for every $s'<s$ the solution map is H\"older continuous of exponent $1-\theta=\frac{s-s'}{s-r}$ from the ball $\{\|(u_0,\tau_0)\|_{H^s}\le\varepsilon\}$, endowed with the $H^s$ (or even the $H^r$) topology, into $C_b([0,\infty);H^{s'})$.
\end{corollary}

\medskip
\noindent\textbf{Remarks on \Cref{thm:unif-dep}.}

\smallskip
\noindent\emph{(a) The setting in which the difference is measured.}
The solutions live in $H^s$, but the Lipschitz estimate \eqref{eq:unif-dep} is stated in $H^r$, $r=s-(2\beta-1)$, on both sides: the solution map is Lipschitz on the small $H^s$-ball \emph{with respect to the $H^r$-topology}, and there is no loss between the two sides of \eqref{eq:unif-dep}. This is the natural analogue of \eqref{eq:FK-stab}, where the data and the solutions are measured in the same space $\dot H^{\frac12}$ while the solutions themselves are more regular. The "loss" of $2\beta-1$ derivatives is the gap between the regularity $H^s$ in which the solutions exist and the regularity $H^r$ in which the map is Lipschitz. It vanishes when $\beta=\frac12$: at the endpoint of the global theory the solution map is Lipschitz in $H^s$ itself. This is consistent with the wave structure \eqref{eq:wave-nonlinear}: when $\beta=\frac12$ the damping $\eta|\xi|^{2\beta}$ and the wave speed $|\xi|$ balance at every frequency, and the hidden velocity dissipation $(-\Delta)^{1-\beta}=(-\Delta)^{\frac12}$ is the strongest possible.

\smallskip
\noindent\emph{(b) Uniformity in time and the explicit constant.}
The constant $\sqrt3\,e^{C_1\varepsilon^2}$ in \eqref{eq:unif-dep} is independent of $t$, and it tends to $\sqrt3$ as $\varepsilon\to0$; the factor $\sqrt3$ comes solely from the equivalence between the Lyapunov functional and the $H^r$-energy, while $e^{C_1\varepsilon^2}$ is a Gronwall factor whose exponent is the total dissipation budget $\int_0^\infty(\|\nabla u^{(i)}\|_{H^{s-\beta}}^2+\|\Lambda^\beta\tau^{(i)}\|_{H^s}^2)dt$ of the two reference solutions. Uniformity in time is the heart of the matter and is what distinguishes \Cref{thm:unif-dep} from a uniqueness or continuous-dependence statement on a bounded time interval, such as those in \cite{CWZZ}: a bound of the form $\|\nabla u^{(2)}\|_{L^\infty}\|\delta u\|_{L^2}^2$, which is all that is available for the Euler part of the system, produces a Gronwall coefficient that is bounded but \emph{not} integrable on $(0,\infty)$. Every coefficient in our energy inequality is instead a square of one of the $L^2_t$-integrable dissipation norms of \eqref{eq:small-data-bound}; the low-frequency parts of the difference, which no dissipation controls, are handled by a Gagliardo--Nirenberg interpolation that spreads the dissipation over both factors of the quadratic terms.

\smallskip
\noindent\emph{(c) Where the loss comes from.}
The index $r=s+1-2\beta$ is optimal for the energy method: it is dictated by the single term $(\Lambda^r(\delta u\cdot\nabla u^{(2)}),\Lambda^r\delta u)$, in which the reference velocity appears with one derivative. The velocity difference gains $1-\beta$ derivatives from the hidden dissipation and the reference velocity gains $1-\beta$ derivatives from \eqref{eq:small-data-bound}; the term therefore closes with a time-integrable Gronwall coefficient if and only if $r+1\le (s+1-\beta)+(1-\beta)$, that is, $r\le s+1-2\beta$ (see \Cref{rmk:optimal}). This is a statement about the method rather than about the solution map: we do not know whether the map is Lipschitz, or even uniformly continuous, in $H^{r'}$ for some $r'>r$, in particular in $H^s$ itself, when $\beta>\frac12$. A negative answer would require a construction in the spirit of \cite{AlexGerard,Inci} for the coupled system, where the stress dissipation competes with the transport mechanism.

\medskip
\noindent\textbf{Outline of the proof.}
Let $\delta u=u^{(1)}-u^{(2)}$ and $\delta\tau=\tau^{(1)}-\tau^{(2)}$. The linear part of the difference system is the same as that of \eqref{eq:OB_eta}, so $(\delta u,\bbP\nabla\cdot\delta\tau)$ obeys the damped wave equations \eqref{eq:wave-nonlinear} with nonlinearities that are bilinear in the difference and the reference solutions. To capture the dissipation that this structure transfers to $\delta u$, we introduce the Lyapunov functional
\[
\LL(t) = \Hsnorm{\delta u}{r}^2
+ \Hsnorm{\delta\tau}{r}^2
+ 2k\,(\delta u,\,\divg\,\delta\tau)_{H^{r-\beta}},\qquad 0<k\ll1 ,
\]
whose derivative contains, besides the stress dissipation $\eta\|\Lambda^\beta\delta\tau\|_{H^r}^2$, the velocity dissipation $\frac k2\|\nabla\delta u\|_{H^{r-\beta}}^2$ generated by the cross term. The index $r-\beta$ of the cross term is forced by two requirements: the term $-k\|\bbP\nabla\cdot\delta\tau\|_{H^{r-\beta}}^2$ produced by the cross term must be absorbed by the stress dissipation, and the term $k\eta(\Lambda^{2\beta}\delta\tau,\nabla\delta u)_{H^{r-\beta}}$ must be controlled by the two dissipations; both are possible precisely because $\frac12\le\beta\le1$. The thirteen nonlinear terms are then estimated in \Cref{sect:proof} so that each is either absorbed by the dissipation using the smallness of the reference solutions, or bounded by (a dissipation norm)$\times$(the energy)$^{1/2}\times$(a dissipation norm of a reference solution). The last type of bound, after Young's inequality, yields Gronwall coefficients that are integrable on $(0,\infty)$ by \eqref{eq:small-data-bound}, and Gronwall's lemma gives \eqref{eq:unif-dep} and \eqref{eq:unif-dep-diss} with the explicit constants. The only places where care is needed beyond standard commutator and product estimates are the low-frequency parts of the quadratic terms, where the difference appears without derivatives and cannot be controlled by any dissipation; these are handled uniformly in $d=2,3$ and $\frac12\le\beta\le1$ by the Gagliardo--Nirenberg inequality $\|f\|_{L^4}\lesssim\|f\|_{L^2}^{1-d/4}\|\nabla f\|_{L^2}^{d/4}$.

\medskip
\noindent\textbf{Organization of the paper.}
\Cref{sect:prelim} collects the notation and the analytic tools (embeddings, product and commutator estimates) with the precise index restrictions they require. \Cref{sect:proof} contains the proof of \Cref{thm:unif-dep} and of \Cref{cor:holder}, and ends with \Cref{rmk:optimal} on the optimality of $r$.

\vskip .2in
\section{Preliminaries}
\label{sect:prelim}
\setcounter{equation}{0}

Throughout the paper $d\in\{2,3\}$, $\frac12\le\beta\le1$, $s>1+\frac d2$, and
\[
r:=s-(2\beta-1)=s+1-2\beta,\qquad \sigma:=r-\beta=s+1-3\beta .
\]
We fix once and for all the following (equivalent) definition of the inhomogeneous Sobolev norm and inner product: for $\rho\ge0$,
\begin{equation}\label{eq:norm-def}
\Hsnorm{f}{\rho}^2:=\Lpnorm{f}{2}^2+\Lpnorm{\Lambda^\rho f}{2}^2,\qquad
\ipHs{f}{g}{\rho}:=\ipltwo{f}{g}+\ipltwo{\Lambda^\rho f}{\Lambda^\rho g}.
\end{equation}
For vector or matrix-valued functions the $L^2$ inner product is the sum over all components. With this convention, Plancherel's theorem and the elementary inequality $|\xi|^{2b}\le|\xi|^{2a}+|\xi|^{2c}$ for $0\le a\le b\le c$ give
\begin{equation}\label{eq:interp}
\Lpnorm{\Lambda^b f}{2}^2\le \Lpnorm{\Lambda^a f}{2}^2+\Lpnorm{\Lambda^c f}{2}^2,\qquad 0\le a\le b\le c,
\end{equation}
and in particular $\|\Lambda^bf\|_{L^2}\le\|f\|_{H^c}$ for $0\le b\le c$ and $\|\Lambda^b f\|_{L^2}\le \|\Lambda^a f\|_{H^{c-a}}$ for $a\le b\le c$. We shall repeatedly use the following consequences of $\frac12\le\beta\le1$ and $s>1+\frac d2$, all of which are immediate:
\begin{equation}\label{eq:indices}
\begin{gathered}
r>\tfrac d2,\qquad s-\beta>\tfrac d2,\qquad \sigma>0,\qquad
1\le r\le r+1-\beta,\\
\beta\le r-\beta+1\le r+\beta,\qquad
2\beta\le r+\beta,\qquad 0\le r+1-2\beta\le r .
\end{gathered}
\end{equation}
(For instance $r>1+\frac d2-2\beta+1\ge\frac d2$, $\sigma>\frac d2-1\ge0$, and $1-\beta\le\beta$.)
The Leray projection $\bbP=I-\nabla\Delta\inv\divg$ is a Fourier multiplier with matrix symbol $I-\xi\otimes\xi/|\xi|^2$; it is an orthogonal projection on $L^2$ and on every $H^\rho$ with the inner product \eqref{eq:norm-def}, it commutes with $\Lambda^\rho$, and $\bbP f=f$ when $\divg f=0$.

We collect the analytic tools in the next lemma. All constants denoted by $C$ or $C_0$ depend only on $d,s,\beta$ (never on $\eta$, $t$ or the solutions), and $A\lesssim B$ means $A\le C_0B$.

\begin{lemma}\label{lem:tools}
Let $d\in\{2,3\}$.
\begin{enumerate}[(i)]
\item (Sobolev) If $\rho>\frac d2$, then $\|f\|_{L^\infty}\lesssim\|f\|_{H^\rho}$ and $\|f\|_{L^p}\lesssim\|f\|_{H^\rho}$ for every $2\le p\le\infty$. If $0\le a<\frac d2$ and $\frac1p=\frac12-\frac ad$, then $\|f\|_{L^p}\lesssim\|\Lambda^a f\|_{L^2}$.
\item (Gagliardo--Nirenberg) $\|f\|_{L^4}\lesssim\|f\|_{L^2}^{1-\frac d4}\|\nabla f\|_{L^2}^{\frac d4}$.
\item (Products) If $\rho>\frac d2$ and $0\le\mu\le\rho$, then $\|fg\|_{H^\mu}\lesssim\|f\|_{H^\rho}\|g\|_{H^\mu}$; in particular $H^\rho$ is an algebra.
\item (Kato--Ponce commutator \cite{KatoPonce}) If $\rho\ge1$ and $2\le p,q\le\infty$ with $\frac1p+\frac1q=\frac12$, then
\[
\|[\Lambda^\rho,f\cdot\nabla]g\|_{L^2}\lesssim
\|\nabla f\|_{L^\infty}\|\Lambda^{\rho}g\|_{L^2}
+\|\Lambda^\rho f\|_{L^p}\|\nabla g\|_{L^q},
\]
where $[\Lambda^\rho,f\cdot\nabla]g=\Lambda^\rho(f\cdot\nabla g)-f\cdot\nabla\Lambda^\rho g$.
\end{enumerate}
\end{lemma}

In the commutator estimate we shall use one fixed pair of exponents, chosen so that $\Lambda^r$ of a reference velocity can be placed in $L^2_t$ through \eqref{eq:small-data-bound} and $\nabla$ of the difference can be placed in the hidden dissipation $H^\sigma$. Set
\begin{equation}\label{eq:ab}
b:=\min\Bigl\{\sigma,\,\tfrac{d-1}{2}\Bigr\},\qquad a:=\tfrac d2-b,\qquad
\tfrac1p:=\tfrac12-\tfrac ad,\qquad \tfrac1q:=\tfrac12-\tfrac bd .
\end{equation}
Then $0<b<\frac d2$, $a=\max\{\frac d2-\sigma,\frac12\}$ satisfies $\frac12\le a\le\beta$ (because $\frac d2-\sigma<\beta$ is equivalent to $r>\frac d2$) and $a<\frac d2$ (for $d=2$ this is $a<1$, which holds because $\sigma>0$), so that $2<p,q<\infty$, $\frac1p+\frac1q=\frac12$, and by \Cref{lem:tools}(i) and \eqref{eq:interp},
\begin{equation}\label{eq:pq-use}
\|\Lambda^r f\|_{L^p}\lesssim\|\Lambda^{r+a}f\|_{L^2}\le\|\nabla f\|_{H^{s-\beta}},\qquad
\|\nabla g\|_{L^q}\lesssim\|\Lambda^{b}\nabla g\|_{L^2}\le\|\nabla g\|_{H^{\sigma}},
\end{equation}
where we used $1\le r+a\le r+\beta=s+1-\beta$ and $0\le b\le\sigma$.

\vskip .2in
\section{Proof of \Cref{thm:unif-dep}}
\label{sect:proof}
\setcounter{equation}{0}

Let $(u^{(1)},\tau^{(1)})$ and $(u^{(2)},\tau^{(2)})$ be the solutions of \Cref{prop:small-data} with $\|u^{(i)}_0\|_{H^s}+\|\tau^{(i)}_0\|_{H^s}\le\varepsilon$. We write $\du=u\one-u\two$, $\dtau=\tau\one-\tau\two$, $\delta p=p\one-p\two$, and introduce the energy and dissipation of the difference,
\[
\EE:=\Hsnorm{\du}{r}^2+\Hsnorm{\dtau}{r}^2,\qquad
\YY:=\Hsnorm{\nabla\du}{\sigma}^2,\qquad
\ZZd:=\Hsnorm{\Lambda^\beta\dtau}{r}^2,\qquad
\DD:=\YY+\ZZd,
\]
and the quantities controlled by \Cref{prop:small-data},
\[
M:=\sup_{t\ge0}\sum_{i=1}^2\bigl(\|u^{(i)}(t)\|_{H^s}+\|\tau^{(i)}(t)\|_{H^s}\bigr),\qquad
F(t):=\sum_{i=1}^2\Bigl(\|\nabla u^{(i)}(t)\|_{H^{s-\beta}}^2+\|\Lambda^\beta\tau^{(i)}(t)\|_{H^s}^2\Bigr),
\]
so that, by \eqref{eq:small-data-bound},
\begin{equation}\label{eq:MF}
M\le 2\sqrt{2C_*}\,\varepsilon,\qquad \int_0^\infty F(t)\,dt\le 2C_*\varepsilon^2 .
\end{equation}
By \eqref{eq:interp} and \eqref{eq:indices}, the dissipation norms control the following quantities, which is how they will enter the estimates:
\begin{equation}\label{eq:diss-control}
\begin{gathered}
\Lpnorm{\nabla\du}{2}^2,\ \Lpnorm{\Lambda^r\du}{2}^2,\ \Lpnorm{\Lambda^{r+1-\beta}\du}{2}^2\le\YY,\qquad
\Lpnorm{\nabla\dtau}{2}^2,\ \Lpnorm{\Lambda^r\dtau}{2}^2,\ \Lpnorm{\Lambda^{r+\beta}\dtau}{2}^2\le\ZZd,\\
\Hsnorm{\nabla\dtau}{\sigma}^2\le2\ZZd,\qquad \Hsnorm{\Lambda^{2\beta}\dtau}{\sigma}^2\le2\ZZd,
\end{gathered}
\end{equation}
and likewise, for the reference solutions ($i=1,2$),
\begin{equation}\label{eq:ref-control}
\begin{gathered}
\Lpnorm{\nabla u^{(i)}}{2}\le\min\{M,\|\nabla u^{(i)}\|_{H^{s-\beta}}\},\qquad
\Lpnorm{\nabla \tau^{(i)}}{2}\le\min\{M,\|\Lambda^\beta \tau^{(i)}\|_{H^{s}}\},\\
\Hsnorm{\nabla\tau^{(i)}}{s-\beta}^2\le 2\|\Lambda^\beta\tau^{(i)}\|_{H^s}^2 ,
\end{gathered}
\end{equation}
where for the last inequality we used $\beta\le1\le s+\beta$ and $\beta\le s+1-\beta\le s+\beta$. Finally, since $s-\beta>\frac d2$, $r>\frac d2$ and $r\le s$, \Cref{lem:tools}(i) gives
\begin{equation}\label{eq:Linfty}
\|\nabla u^{(i)}\|_{L^\infty}\lesssim\|\nabla u^{(i)}\|_{H^{s-\beta}},\qquad
\|u^{(i)}\|_{H^r}+\|\tau^{(i)}\|_{H^r}\le M,\qquad
\|\du\|_{L^\infty}+\|\dtau\|_{L^\infty}\lesssim\EE^{\frac12}.
\end{equation}

\subsection{The difference system and the energy identities}
The pair $(\du,\dtau)$ satisfies
\begin{equation}\label{diff}
\begin{cases}
\partial_t \du+u^{(1)}\cdot\nabla \du+\du\cdot\nabla u^{(2)}+\nabla \delta p= \nabla\cdot \dtau,\\
\partial_t \dtau +u^{(1)}\cdot\nabla\dtau+\du \cdot\nabla\tau^{(2)}+\eta\Lambda^{2\beta}\dtau+Q(\tau^{(1)},\nabla\du)+Q(\dtau,\nabla u^{(2)})=D(\du),\\
\nabla\cdot\du=0,\\
\du(x,0)=u^{(1)}_0-u_0^{(2)},\quad \dtau(x,0)=\tau^{(1)}_0-\tau_0^{(2)}.
\end{cases}
\end{equation}
We use three cancellations. Since $\divg u\one=\divg\du=0$, $(u\one\cdot\nabla f,f)_{L^2}=0$ and $(\nabla\delta p,\du)_{L^2}=0$; and since $\dtau$ is symmetric,
\begin{equation}\label{eq:cancel}
\ipltwo{\nabla\cdot\dtau}{\du}+\ipltwo{D(\du)}{\dtau}
=-\ipltwo{\dtau}{\nabla\du}+\ipltwo{\dtau}{D(\du)}=0,
\end{equation}
where $(\nabla\du)_{ij}=\partial_i\du_j$. The same identities hold after applying $\Lambda^r$ to \eqref{diff}, because $\Lambda^r$ commutes with $\nabla$ and preserves symmetry and the divergence-free condition. Taking the $L^2$ inner product of \eqref{diff} with $(\du,\dtau)$ and of $\Lambda^r$\eqref{diff} with $(\Lambda^r\du,\Lambda^r\dtau)$, and adding, we obtain
\begin{equation}
\label{eq:energy}
\frac{1}{2}\frac{d}{dt}\EE+\eta\,\ZZd=\sum_{i=1}^{6}J_i ,
\end{equation}
where
\begin{align*}
&J_1 = -\bigl(\Lambda^r(u^{(1)}\cdot\nabla\du),\,\Lambda^r\du\bigr)_{L^2}
=-\bigl([\Lambda^r,u\one\cdot\nabla]\du,\,\Lambda^r\du\bigr)_{L^2},\\
&J_2= -\bigl(\Lambda^r(u^{(1)}\cdot\nabla\dtau),\,\Lambda^r\dtau\bigr)_{L^2}
=-\bigl([\Lambda^r,u\one\cdot\nabla]\dtau,\,\Lambda^r\dtau\bigr)_{L^2},\\
&J_3= -\ipHs{\du\cdot\nabla u^{(2)}}{\du}{r},\qquad
J_4= -\ipHs{\du\cdot\nabla\tau^{(2)}}{\dtau}{r}, \\
&J_5= -\ipHs{Q(\tau^{(1)},\nabla\du)}{\dtau}{r},\qquad
J_6= -\ipHs{Q(\dtau,\nabla u^{(2)})}{\dtau}{r}.
\end{align*}

Next we exploit the wave structure \eqref{eq:wave-nonlinear} to generate a dissipative term for $\du$, by computing the time derivative of the cross term $(\du,\nabla\cdot\dtau)_{H^\sigma}$. Since $\partial_t\du$ is divergence-free and $\bbP$ is an orthogonal projection on $H^\sigma$,
\[
\ipHs{\partial_t\du}{\nabla\cdot\dtau}{\sigma}=\ipHs{\partial_t\du}{\bbP\nabla\cdot\dtau}{\sigma},
\]
and applying $\bbP$ to the first equation in \eqref{diff} gives $\partial_t\du=\bbP\nabla\cdot\dtau-\bbP(u\one\cdot\nabla\du)-\bbP(\du\cdot\nabla u\two)$. On the other hand, integrating by parts, $(\du,\nabla\cdot\partial_t\dtau)_{H^\sigma}=-(\nabla\du,\partial_t\dtau)_{H^\sigma}$, and since $\divg\du=0$,
\[
\ipHs{\nabla\du}{D(\du)}{\sigma}=\tfrac12\Hsnorm{\nabla\du}{\sigma}^2+\tfrac12\ipHs{\partial_i\du_j}{\partial_j\du_i}{\sigma}
=\tfrac12\Hsnorm{\nabla\du}{\sigma}^2+\tfrac12\Hsnorm{\divg\du}{\sigma}^2=\tfrac12\Hsnorm{\nabla\du}{\sigma}^2 .
\]
Using the second equation in \eqref{diff} we therefore find
\begin{equation}
\label{eq:cross}
\begin{split}
&\frac{d}{dt}\ipHs{\du}{\divg\dtau}{\sigma} +\frac12\Hsnorm{\nabla\du}{\sigma}^2- \Hsnorm{\bbP \divg\dtau}{\sigma}^2 \\
&\quad=-\ipHs{\bbP(u^{(1)}\cdot\nabla\du)}{\divg\dtau}{\sigma}
-\ipHs{\bbP(\du\cdot\nabla u^{(2)})}{\divg\dtau}{\sigma}
+\ipHs{u^{(1)}\cdot\nabla\dtau}{\nabla\du}{\sigma}\\
&\qquad +\ipHs{\du\cdot\nabla\tau^{(2)}}{\nabla\du}{\sigma}
+\ipHs{Q(\tau^{(1)},\nabla\du)}{\nabla\du}{\sigma}
+\ipHs{Q(\dtau,\nabla u^{(2)})}{\nabla\du}{\sigma}\\
&\qquad +\eta\ipHs{\Lambda^{2\beta}\dtau}{\nabla\du}{\sigma}.
\end{split}
\end{equation}
The two terms $\frac12\|\nabla\du\|_{H^\sigma}^2$ and $-\|\bbP\nabla\cdot\dtau\|_{H^\sigma}^2$ on the left of \eqref{eq:cross} are the difference-level manifestation of the wave structure: the first is the hidden dissipation of $\du$, the second must be paid for with the stress dissipation. For $k>0$ to be chosen, define the Lyapunov functional
\[
\LL(t):=\Hsnorm{\du}{r}^2+\Hsnorm{\dtau}{r}^2+2k\ipHs{\du}{\divg\dtau}{\sigma}.
\]
Then \eqref{eq:energy}$+k\times$\eqref{eq:cross} reads
\begin{equation}
\label{eq:energy-deriv}
\frac12\frac{d}{dt}\LL +\eta\,\ZZd- k\Hsnorm{\bbP\divg\dtau}{\sigma}^2+\frac k2\YY = \sum_{i=1}^{13}J_i,
\end{equation}
where $J_1,\dots,J_6$ are as above and
\begin{align*}
&J_7= -k\ipHs{\bbP(u^{(1)}\cdot\nabla\du)}{\divg\dtau}{\sigma},\qquad
J_8= -k\ipHs{\bbP(\du\cdot\nabla u^{(2)})}{\divg\dtau}{\sigma}, \\
&J_9= k\ipHs{u^{(1)}\cdot\nabla\dtau}{\nabla\du}{\sigma},\qquad\qquad\ \
J_{10}= k\ipHs{\du\cdot\nabla\tau^{(2)}}{\nabla\du}{\sigma},\\
&J_{11}= k\ipHs{Q(\tau^{(1)},\nabla\du)}{\nabla\du}{\sigma}, \qquad\quad\
J_{12}= k\ipHs{Q(\dtau,\nabla u^{(2)})}{\nabla\du}{\sigma}, \\
&J_{13}= k\eta\ipHs{\Lambda^{2\beta}\dtau}{\nabla\du}{\sigma}.
\end{align*}
Two features of \eqref{eq:energy-deriv} are worth recording before we estimate the $J_i$. First, since $\bbP$ has norm one and $\|\nabla\dtau\|_{H^\sigma}^2\le2\ZZd$ by \eqref{eq:diss-control},
\begin{equation}\label{eq:Pdiv}
k\Hsnorm{\bbP\divg\dtau}{\sigma}^2\le 2k\,\ZZd ,
\end{equation}
which is where $\beta\ge\frac12$ enters (through $r-\beta+1\le r+\beta$). Second, by \eqref{eq:interp} and \eqref{eq:indices},
\begin{equation}\label{eq:cross-small}
\begin{split}
\bigl|2k\ipHs{\du}{\divg\dtau}{\sigma}\bigr|
&=2k\bigl|\ipltwo{\du}{\divg\dtau}+\ipltwo{\Lambda^r\du}{\Lambda^{r-2\beta}\divg\dtau}\bigr|\\
&\le 2k\Hsnorm{\du}{r}\bigl(\Lpnorm{\nabla\dtau}{2}^2+\Lpnorm{\Lambda^{r+1-2\beta}\dtau}{2}^2\bigr)^{\frac12}
\le 2\sqrt2\,k\Hsnorm{\du}{r}\Hsnorm{\dtau}{r},
\end{split}
\end{equation}
so that, as soon as $k\le\frac14$,
\begin{equation}\label{eq:equiv}
\tfrac12\,\EE\le\LL\le\tfrac32\,\EE .
\end{equation}

\subsection{Estimates of the nonlinear terms}
We now bound $J_1,\dots,J_{13}$. Each bound is of one of three types: (A) $\lesssim M\,\DD$, absorbed by the dissipation once $\varepsilon$ is small; (B) $\lesssim G\,\EE^{1-\gamma}\DD^{\gamma}$ with $\gamma\in(0,1)$ and a coefficient $G$ built from the reference solutions, such that after Young's inequality the resulting coefficient of $\EE$ is bounded by a multiple of $F$, hence integrable on $(0,\infty)$; (C) $\le\sqrt2\,k\eta\,\YY^{\frac12}\ZZd^{\frac12}$, the linear term $J_{13}$. Throughout, $C_0$ denotes a constant depending only on $d,s,\beta$.

\medskip\noindent\emph{The commutator terms $J_1$, $J_2$.}
By \Cref{lem:tools}(iv) with the exponents \eqref{eq:ab}, then \eqref{eq:Linfty}, \eqref{eq:pq-use} and \eqref{eq:diss-control},
\begin{align*}
|J_1| &\lesssim \bigl(\Lpnorm{\nabla u^{(1)}}{\infty}\Lpnorm{\Lambda^r \du}{2} + \Lpnorm{\Lambda^r u^{(1)}}{p}\Lpnorm{\nabla \du}{q} \bigr) \Lpnorm{\Lambda^r \du}{2}\\
&\lesssim \Hsnorm{\nabla u\one}{s-\beta}\,\YY^{\frac12}\,\Lpnorm{\Lambda^r \du}{2}
\lesssim \Hsnorm{\nabla u\one}{s-\beta}\,\EE^{\frac12}\YY^{\frac12}.
\end{align*}
Note that we placed \emph{one} factor $\|\Lambda^r\du\|_{L^2}$ in the dissipation $\YY$ (which is possible since $1\le r\le r+1-\beta$) and the other in the energy; this is what makes the coefficient $\|\nabla u\one\|_{H^{s-\beta}}$ appear to the first power in a bound of type (B), rather than as a coefficient of $\EE$ alone, which would not be integrable in time. In the same way, using $\|\Lambda^r\dtau\|_{L^2}^2\le\ZZd$ and $\|\Lambda^b\nabla\dtau\|_{L^2}^2\le\ZZd$ (recall $\beta\le1+b\le r+\beta$ since $b\le\sigma$),
\begin{align*}
|J_2| &\lesssim \bigl(\Lpnorm{\nabla u\one}{\infty}\Lpnorm{\Lambda^r \dtau}{2} + \Lpnorm{\Lambda^r u\one}{p}\Lpnorm{\nabla \dtau}{q} \bigr) \Lpnorm{\Lambda^r \dtau}{2}\\
&\lesssim \Hsnorm{\nabla u\one}{s-\beta}\,\EE^{\frac12}\ZZd^{\frac12}.
\end{align*}

\medskip\noindent\emph{The terms $J_3$--$J_6$: high-frequency parts.}
We split $J_i=J_{i,\mathrm{low}}+J_{i,\mathrm{high}}$, $i=3,\dots,6$, according to the two inner products in \eqref{eq:norm-def}. For $J_{3,\mathrm{high}}$ we move $1-\beta$ derivatives onto the second factor, where they are paid for by $\YY$, and use the algebra property of $H^{s-\beta}=H^{r+\beta-1}$ (\Cref{lem:tools}(iii), $s-\beta>\frac d2$) together with $r+\beta-1\le r$:
\begin{align*}
|J_{3,\mathrm{high}}| &=\bigl|\ipltwo{\Lambda^{r+\beta-1}(\du\cdot\nabla u^{(2)})}{\Lambda^{r+1-\beta} \du}\bigr|\\
&\lesssim \Hsnorm{\du}{s-\beta}\Hsnorm{\nabla u\two}{s-\beta}\Lpnorm{\Lambda^{r+1-\beta} \du}{2}
\lesssim \Hsnorm{\nabla u\two}{s-\beta}\,\EE^{\frac12}\YY^{\frac12}.
\end{align*}
For $J_{4,\mathrm{high}}$, $J_{5,\mathrm{high}}$ and $J_{6,\mathrm{high}}$ we move $\beta$ derivatives onto $\dtau$ and use \Cref{lem:tools}(iii) with $\rho=r$, $\mu=\sigma$ (note $r>\frac d2$ and $0<\sigma\le r$), \eqref{eq:ref-control} and \eqref{eq:Linfty}:
\begin{align*}
|J_{4,\mathrm{high}}| &\le \Hsnorm{\du\cdot\nabla\tau\two}{\sigma}\Lpnorm{\Lambda^{r+\beta}\dtau}{2}
\lesssim \Hsnorm{ \du}{r}\Hsnorm{\nabla\tau^{(2)}}{\sigma}\ZZd^{\frac12}
\lesssim \Hsnorm{\Lambda^\beta\tau\two}{s}\,\EE^{\frac12}\ZZd^{\frac12},\\
|J_{5,\mathrm{high}}| &\le \Hsnorm{Q(\tau\one,\nabla\du)}{\sigma}\Lpnorm{\Lambda^{r+\beta}\dtau}{2}
\lesssim \Hsnorm{\tau^{(1)}}{r}\Hsnorm{\nabla\du}{\sigma}\ZZd^{\frac12}
\le C_0M\,\YY^{\frac12}\ZZd^{\frac12},\\
|J_{6,\mathrm{high}}| &\le \Hsnorm{Q(\dtau,\nabla u\two)}{\sigma}\Lpnorm{\Lambda^{r+\beta}\dtau}{2}
\lesssim \Hsnorm{\dtau}{r}\Hsnorm{\nabla u^{(2)}}{\sigma}\ZZd^{\frac12}
\lesssim \Hsnorm{\nabla u\two}{s-\beta}\,\EE^{\frac12}\ZZd^{\frac12},
\end{align*}
where we used that $Q$ is bilinear with constant coefficients, so that $\|Q(f,\nabla g)\|_{H^\mu}\lesssim\|f\,\nabla g\|_{H^\mu}$ componentwise, and $\sigma\le s-\beta$.

\medskip\noindent\emph{The terms $J_3$--$J_6$: low-frequency parts.}
The low-frequency parts are the only place where the difference appears without any derivative, and hence cannot be placed in $\YY$ or $\ZZd$ directly; in dimension two this is a genuine obstruction to a uniform-in-time estimate, since $\dot H^1(\R^2)\not\hookrightarrow L^\infty$. We use instead the Gagliardo--Nirenberg inequality of \Cref{lem:tools}(ii), which spreads the dissipation over the two factors of the difference, and place the reference solution in $L^2$, where it is controlled by the dissipation norms through \eqref{eq:ref-control}. With $\gamma_d:=\frac d4\in\{\frac12,\frac34\}$,
\begin{align*}
|J_{3,\mathrm{low}}|&=\bigl|\ipltwo{\du\cdot\nabla u\two}{\du}\bigr|
\le\Lpnorm{\nabla u\two}{2}\Lpnorm{\du}{4}^2\\
&\lesssim\Lpnorm{\nabla u\two}{2}\Lpnorm{\du}{2}^{2-2\gamma_d}\Lpnorm{\nabla\du}{2}^{2\gamma_d}
\le C_0\Lpnorm{\nabla u\two}{2}\,\EE^{1-\gamma_d}\YY^{\gamma_d},\\
|J_{4,\mathrm{low}}|&=\bigl|\ipltwo{\du\cdot\nabla \tau\two}{\dtau}\bigr|
\le\Lpnorm{\nabla \tau\two}{2}\Lpnorm{\du}{4}\Lpnorm{\dtau}{4}\\
&\lesssim\Lpnorm{\nabla \tau\two}{2}\,\EE^{1-\gamma_d}(\YY\ZZd)^{\frac{\gamma_d}{2}}
\le C_0\Lpnorm{\nabla \tau\two}{2}\,\EE^{1-\gamma_d}\DD^{\gamma_d},\\
|J_{6,\mathrm{low}}|&=\bigl|\ipltwo{Q(\dtau,\nabla u\two)}{\dtau}\bigr|
\lesssim\Lpnorm{\nabla u\two}{2}\Lpnorm{\dtau}{4}^2
\le C_0\Lpnorm{\nabla u\two}{2}\,\EE^{1-\gamma_d}\ZZd^{\gamma_d},\\
|J_{5,\mathrm{low}}|&=\bigl|\ipltwo{Q(\tau\one,\nabla\du)}{\dtau}\bigr|
\lesssim\Lpnorm{\tau\one}{4}\Lpnorm{\nabla\du}{2}\Lpnorm{\dtau}{4}\\
&\lesssim\Lpnorm{\tau\one}{2}^{1-\gamma_d}\Lpnorm{\nabla\tau\one}{2}^{\gamma_d}\,
\YY^{\frac12}\,\EE^{\frac{1-\gamma_d}2}\ZZd^{\frac{\gamma_d}2}\\
&\le C_0 M^{1-\gamma_d}\Lpnorm{\nabla\tau\one}{2}^{\gamma_d}\,\EE^{\frac{1-\gamma_d}2}\DD^{\frac{1+\gamma_d}{2}},
\end{align*}
where we used $\|\nabla\du\|_{L^2}^2\le\YY$, $\|\nabla\dtau\|_{L^2}^2\le\ZZd$ and $\sqrt{\YY\ZZd}\le\DD$. When $d=2$ and $\beta=1$ the bounds for $J_{4,\mathrm{low}}$ and $J_{5,\mathrm{low}}$ are the classical Ladyzhenskaya estimates; the point of the formulation above is that it treats $d=2,3$ and $\frac12\le\beta\le1$ simultaneously and never uses $\|\nabla u\two\|_{L^\infty}$ or $\|\tau\one\|_{L^\infty}$ as a coefficient of the energy.

\medskip\noindent\emph{The cross terms $J_7$--$J_{12}$.}
These are bounded directly in $H^\sigma$ by \Cref{lem:tools}(iii) (with $\rho=r$, $\mu=\sigma$), \eqref{eq:diss-control}, \eqref{eq:ref-control} and \eqref{eq:Linfty}; the projection $\bbP$ is dropped by moving it to the other factor and using $\|\bbP\|\le1$:
\begin{align*}
|J_7| &\le k\Hsnorm{u\one\cdot\nabla\du}{\sigma}\Hsnorm{\nabla\dtau}{\sigma}
\lesssim k\Hsnorm{u\one}{r}\Hsnorm{\nabla\du}{\sigma}\ZZd^{\frac12}\le C_0kM\,\YY^{\frac12}\ZZd^{\frac12},\\
|J_8| &\le k\Hsnorm{\du\cdot\nabla u\two}{\sigma}\Hsnorm{\nabla\dtau}{\sigma}
\lesssim k\Hsnorm{\du}{r}\Hsnorm{\nabla u\two}{\sigma}\ZZd^{\frac12}\le C_0k\Hsnorm{\nabla u\two}{s-\beta}\,\EE^{\frac12}\ZZd^{\frac12},\\
|J_9| &\le k\Hsnorm{u\one\cdot\nabla\dtau}{\sigma}\Hsnorm{\nabla\du}{\sigma}
\lesssim k\Hsnorm{u\one}{r}\Hsnorm{\nabla\dtau}{\sigma}\YY^{\frac12}\le C_0kM\,\YY^{\frac12}\ZZd^{\frac12},\\
|J_{10}| &\le k\Hsnorm{\du\cdot\nabla\tau\two}{\sigma}\Hsnorm{\nabla\du}{\sigma}
\lesssim k\Hsnorm{\du}{r}\Hsnorm{\nabla\tau\two}{\sigma}\YY^{\frac12}\le C_0k\Hsnorm{\Lambda^\beta\tau\two}{s}\,\EE^{\frac12}\YY^{\frac12},\\
|J_{11}| &\le k\Hsnorm{Q(\tau\one,\nabla\du)}{\sigma}\Hsnorm{\nabla\du}{\sigma}
\lesssim k\Hsnorm{\tau\one}{r}\YY\le C_0kM\,\YY,\\
|J_{12}| &\le k\Hsnorm{Q(\dtau,\nabla u\two)}{\sigma}\Hsnorm{\nabla\du}{\sigma}
\lesssim k\Hsnorm{\dtau}{r}\Hsnorm{\nabla u\two}{\sigma}\YY^{\frac12}\le C_0k\Hsnorm{\nabla u\two}{s-\beta}\,\EE^{\frac12}\YY^{\frac12}.
\end{align*}

\medskip\noindent\emph{The linear term $J_{13}$.}
By \eqref{eq:diss-control}, $\|\Lambda^{2\beta}\dtau\|_{H^\sigma}^2\le2\ZZd$ (this uses $\beta\le2\beta\le r+\beta$), hence
\[
|J_{13}|\le k\eta\Hsnorm{\Lambda^{2\beta}\dtau}{\sigma}\Hsnorm{\nabla\du}{\sigma}\le\sqrt2\,k\eta\,\ZZd^{\frac12}\YY^{\frac12}
\le\frac{\eta}{8}\ZZd+4k^2\eta\,\YY .
\]

\subsection{Choice of $k$ and $\varepsilon_0$, absorption, and Gronwall's inequality}
We fix
\begin{equation}\label{eq:k}
k:=\tfrac1{32}\min\{\eta,\ \eta\inv,\ 1\},
\end{equation}
so that $k\le\frac14$, $2k\le\frac\eta4$ and $4k^2\eta\le\frac k8$. Then \eqref{eq:Pdiv} and the bound on $J_{13}$ turn \eqref{eq:energy-deriv} into
\begin{equation}\label{eq:step1}
\frac12\frac{d}{dt}\LL+\frac{5\eta}{8}\ZZd+\frac{3k}{8}\YY\le\sum_{i=1}^{12}|J_i| .
\end{equation}
Since $k\le\eta$, we have $\frac{\eta}{8}\ZZd+\frac k8\YY\ge\frac k8\DD$. We shall bound $\sum_{i\le12}|J_i|$ by $\frac{k}{8}\DD$ plus an integrable multiple of $\EE$, allotting $\frac{k}{24}\DD$ to each of the three groups below.

\smallskip\noindent\emph{Type (A).}
The terms $J_{5,\mathrm{high}},J_7,J_9,J_{11}$ are bounded by $C_0M(1+2k)\YY^{\frac12}\ZZd^{\frac12}+C_0kM\YY\le 2C_0M\,\DD$. We now fix
\begin{equation}\label{eq:eps0}
\varepsilon_0:=\min\Bigl\{\varepsilon_*,\ \frac{k}{96\,C_0\sqrt{2C_*}},\ \frac{1}{2\sqrt{2C_*}}\Bigr\},
\end{equation}
so that, by \eqref{eq:MF}, $2C_0M\le\frac{k}{24}$ and $M\le1$ whenever $\varepsilon\le\varepsilon_0$. Hence the type (A) terms are bounded by $\frac k{24}\DD$.

\smallskip\noindent\emph{Type (B), exponent $\frac12$.}
The remaining terms among $J_1,J_2,J_{3,\mathrm{high}},J_{4,\mathrm{high}},J_{6,\mathrm{high}},J_8,J_{10},J_{12}$ are all of the form $C_0G\,\EE^{\frac12}\DD^{\frac12}$ with $G\in\{\|\nabla u^{(i)}\|_{H^{s-\beta}},\|\Lambda^\beta\tau^{(i)}\|_{H^s}\}$, so $G^2\le F$. By Young's inequality ($C_0G\EE^{\frac12}\DD^{\frac12}\le\frac{k}{192}\DD+\frac{48C_0^2}{k}G^2\EE$ for each of the eight terms) their sum is bounded by
\[
\frac k{24}\DD+\frac{C_0}{k}\,F\,\EE .
\]

\smallskip\noindent\emph{Type (B), other exponents.}
For $J_{3,\mathrm{low}},J_{4,\mathrm{low}},J_{6,\mathrm{low}}$ we have $C_0G_0\EE^{1-\gamma_d}\DD^{\gamma_d}$ with $G_0\in\{\|\nabla u\two\|_{L^2},\|\nabla\tau\two\|_{L^2}\}$, and Young's inequality with exponents $\frac1{\gamma_d}$ and $\frac1{1-\gamma_d}$ gives, for $\delta>0$,
\[
C_0G_0\EE^{1-\gamma_d}\DD^{\gamma_d}\le\delta\DD+C\,\delta^{-\frac{\gamma_d}{1-\gamma_d}}\,G_0^{\frac{1}{1-\gamma_d}}\EE .
\]
Here $\frac1{1-\gamma_d}=2$ if $d=2$ and $=4$ if $d=3$; in the latter case $G_0^4\le M^2G_0^2\le G_0^2$ by \eqref{eq:ref-control}, and in both cases $G_0^2\le F$. For $J_{5,\mathrm{low}}$ we have $C_0M^{1-\gamma_d}\|\nabla\tau\one\|_{L^2}^{\gamma_d}\EE^{\frac{1-\gamma_d}{2}}\DD^{\frac{1+\gamma_d}2}$ and Young's inequality with exponents $\frac{2}{1+\gamma_d}$ and $\frac{2}{1-\gamma_d}$ gives
\[
|J_{5,\mathrm{low}}|\le\delta\DD+C\delta^{-\frac{1+\gamma_d}{1-\gamma_d}}M^{2}\Lpnorm{\nabla\tau\one}{2}^{\frac{2\gamma_d}{1-\gamma_d}}\EE ,
\]
where $\frac{2\gamma_d}{1-\gamma_d}=2$ if $d=2$ and $=6$ if $d=3$; in the latter case $\|\nabla\tau\one\|_{L^2}^6\le M^4\|\nabla\tau\one\|_{L^2}^2$, and in both cases $M\le1$ and $\|\nabla\tau\one\|_{L^2}^2\le F$. Choosing $\delta=\frac{k}{96}$ for each of these four terms, their sum is bounded by
\[
\frac{k}{24}\DD+C_0\,k^{-7}\,F\,\EE .
\]

\smallskip\noindent\emph{Conclusion.}
Inserting the three groups of bounds into \eqref{eq:step1} and recalling $\frac{\eta}{8}\ZZd+\frac k8\YY\ge\frac k8\DD$, we arrive at
\begin{equation}\label{eq:final-ode}
\frac12\frac{d}{dt}\LL+\frac{\eta}{2}\ZZd+\frac{k}{4}\YY\le C_2\,F(t)\,\EE(t),
\qquad C_2:=C_0\,k^{-7},
\end{equation}
where, by \eqref{eq:k}, $C_2$ depends only on $d,s,\beta$ and $\eta$. By \eqref{eq:equiv}, $\EE\le2\LL$, so
\[
\frac{d}{dt}\LL\le 4C_2F(t)\,\LL,
\]
and Gronwall's inequality together with \eqref{eq:equiv} and \eqref{eq:MF} yields, for all $t\ge0$,
\[
\EE(t)\le2\LL(t)\le2\LL(0)\exp\Bigl(4C_2\int_0^tF\Bigr)\le3\,\EE(0)\exp\Bigl(4C_2\int_0^\infty F\Bigr)\le3\,e^{8C_*C_2\,\varepsilon^2}\,\EE(0).
\]
This is \eqref{eq:unif-dep} with $C_1:=8C_*C_2$. Integrating \eqref{eq:final-ode} in time and using $\frac{\eta}{2}\ZZd+\frac k4\YY\ge\frac k4\DD$, $\LL\ge0$ and the bound just obtained,
\[
\frac k4\int_0^\infty\DD\,dt\le\frac12\LL(0)+C_2\int_0^\infty F\,\EE\,dt
\le\frac34\EE(0)+2C_*C_2\varepsilon^2\cdot3e^{C_1\varepsilon^2}\EE(0),
\]
which gives \eqref{eq:unif-dep-diss} after enlarging $C_1$ if necessary. This completes the proof of \Cref{thm:unif-dep}. \qed

\medskip
\begin{proof}[Proof of \Cref{cor:holder}]
For $r\le s'\le s$ and $\theta=\frac{s'-r}{s-r}$, H\"older's inequality on the Fourier side gives $\|\Lambda^{s'}f\|_{L^2}\le\|\Lambda^rf\|_{L^2}^{1-\theta}\|\Lambda^sf\|_{L^2}^\theta$, and then H\"older's inequality for sums of two terms gives $\|f\|_{H^{s'}}\le\|f\|_{H^r}^{1-\theta}\|f\|_{H^s}^{\theta}$ for the norms \eqref{eq:norm-def}. Applying this to $\du(t)$ and $\dtau(t)$, using $\|\du(t)\|_{H^s}^2+\|\dtau(t)\|_{H^s}^2\le 2(\|u\one\|_{H^s}^2+\|u\two\|_{H^s}^2+\|\tau\one\|_{H^s}^2+\|\tau\two\|_{H^s}^2)\le4C_*\varepsilon^2$ and \eqref{eq:unif-dep}, the claim follows.
\end{proof}

\begin{remark}[Optimality of $r$ for the energy method]
\label{rmk:optimal}
We explain in what sense the loss of $2\beta-1$ derivatives is sharp. The index $r=s+1-2\beta$ is forced by the term $J_{3,\mathrm{high}}=-(\Lambda^r(\du\cdot\nabla u^{(2)}),\Lambda^r\du)_{L^2}$, in which the reference velocity carries one derivative and no commutator structure is available. This term contains $2r+1$ derivatives in total. The resources at our disposal are: the energy, which controls $r$ derivatives of $\du$ with a coefficient that is only bounded in time; the hidden dissipation, which controls $r+1-\beta$ derivatives of $\du$ in $L^2_t$; and the reference solution, which by \eqref{eq:small-data-bound} is controlled in $L^\infty_tH^{s}$ and in $L^2_tH^{s+1-\beta}$. Suppose we try to close the estimate at a level $r'$ in place of $r$. There are two ways to handle $J_{3,\mathrm{high}}$ in a uniform-in-time Gronwall argument.
\begin{enumerate}[(1)]
\item If the term is to enter as a multiple of $\EE$, its coefficient must be integrable on $(0,\infty)$, so the reference velocity must be measured in $L^2_tH^{s+1-\beta}$, and then at most one factor of $\du$ can be placed in the dissipation. The most favorable distribution of derivatives is
\[
|J_{3,\mathrm{high}}|\le\Lpnorm{\Lambda^{\rho}(\du\cdot\nabla u^{(2)})}{2}\Lpnorm{\Lambda^{2r'-\rho}\du}{2},\qquad
\rho\le s-\beta,\quad 2r'-\rho\le r'+1-\beta ,
\]
which is possible if and only if $r'\le\rho+1-\beta\le s+1-2\beta$.
\item Alternatively, the term can be absorbed by the dissipation using the smallness of the reference solution, as in the type (A) bounds. Then the reference velocity may be measured in $L^\infty_tH^{s}$, i.e. $\nabla u^{(2)}\in H^{s-1}$, but both factors of $\du$ must now be placed in the dissipation: $\rho\le\min\{s-1,\,r'+1-\beta\}$ and $2r'-\rho\le r'+1-\beta$, which forces $r'\le s-\beta$.
\end{enumerate}
Since $s-\beta\le s+1-2\beta$ for $\beta\le1$, the best index the energy method can reach is $r'=s+1-2\beta$, which is exactly the value of $r$ in \Cref{thm:unif-dep}; the two options coincide only when $\beta=1$. In this sense the loss of $2\beta-1$ derivatives is optimal for the energy method, and it is the precise price of the fact that the velocity dissipation generated by the wave structure \eqref{eq:wave-nonlinear} is of order $2(1-\beta)$ rather than of order $2$.

We emphasize that this is a statement about the method and not about the solution map itself. For $\beta=\frac12$ there is no loss, so the result is sharp. For $\beta>\frac12$, whether the solution map is Lipschitz, or merely uniformly continuous, in $H^{r'}$ for some $r'\in(r,s]$ remains open; a negative answer would require a counterexample in the spirit of \cite{AlexGerard,Inci} for the coupled system, where the stretching mechanism of the transport terms competes with the dissipation transferred from the stress, and we do not pursue this here.
\end{remark}

\vskip .2in
\noindent\textbf{Acknowledgments.} Wu was partially supported by the National Science Foundation of the United States under DMS 2104682 and DMS 2309748.

\end{document}